\documentclass{amsart}%
\usepackage{amssymb}
\usepackage{amsfonts}
\usepackage{amsmath}
\usepackage{graphicx}%
\newtheorem{theorem}{Theorem}
\theoremstyle{plain}

\newtheorem{conjecture}{Conjecture}
\newtheorem{corollary}{Corollary}

\newtheorem{example}{Example}

\newtheorem{proposition}{Proposition}
\newtheorem{remark}{Remark}

\numberwithin{equation}{section}
\begin{document}
\title[ ]{On compact Riemannian manifolds with convex boundary and Ricci curvature
bounded from below}
\author{Xiaodong Wang}
\address{Department of Mathematics, Michigan State University, East Lansing, MI 48824}
\email{xwang@math.msu.edu}

\begin{abstract}
We propose to study positive harmoninc functions satisfying a nonlinear Neuman
condition on a compact Riemannian manifold with nonnegative Ricci curvature
and strictly convex boundary. A precise conjecture is formulated. We discuss
its implications and present some partial results. Related questions are
discussed for compact Riemannian manifolds with positive Ricci curvature and
convex boundary.

\end{abstract}
\maketitle

\section{Introduction}

For a compact Riemannian manifold $\left(  M^{n},g\right)  $ with nonempty
boundary $\Sigma=\partial M$, it is interesting to study connections between
the intrinsic geometry $g|_{\Sigma}$ and the extrinsic geometry (the 2nd
fundamental form), under a lower bound for scalar curvature or Ricci
curvature. We refer to \cite{ST1, ST2, WY1, MW} and references therein for
recent works in this direction. Some of these works are motivated by problems
in general relativity, in particular about understanding various definitions
of quasi-local mass. The following fundamental result was proved by Shi and
Tam \cite{ST1}.

\begin{theorem}
\label{ST}Let $\left(  M^{n},g\right)  $ be a compact Riemannian manifold with
scalar curvature $R\geq0$ and with a connected boundary $\Sigma$. Suppose

\begin{itemize}
\item $M$ is spin,

\item the mean curvature $H$ of $\Sigma$ is positive,

\item there exists an isometric embedding $\iota:\Sigma\rightarrow
\mathbb{R}^{n}$ as a strictly convex hypersurface.
\end{itemize}

Then
\begin{equation}
\int_{\Sigma}H\leq\int_{\Sigma}H_{0}, \label{ste}%
\end{equation}
where $H_{0}$ is the mean curvature of $\iota:\Sigma\rightarrow\mathbb{R}^{n}%
$. Moreover, if equality holds, then $M$ is isometric to the Euclidean domain
enclosed by $\iota:\Sigma\rightarrow\mathbb{R}^{n}$.
\end{theorem}

The right hand side of (\ref{ste}) is determined by the intrinsic geometry of
$\Sigma$. Therefore by the inequality the extrinsic geometry of $\Sigma$ is
constrained by its intrinsic geometry. But the assumption that there is an
isometric embedding of $\Sigma$ into $\mathbb{R}^{n}$ as a strictly convex
hypersurface imposes severe restriction on the kind of intrinsic geometry of
$\Sigma$ for which the theorem is applicable. In the more recent work
\cite{MW} Miao and I proved a slightly different inequality under the stronger
condition $Ric\geq0$, but without any restriction on the intrinsic geometry of
the boundary.

\begin{theorem}
\label{MW}\bigskip Let $\left(  M^{n},g\right)  $ be a compact Riemannian
manifold with $Ric\geq0$ and with a connected boundary $\Sigma$ that has
positive mean curvature $H$. Let $\iota:\Sigma\rightarrow\mathbb{R}^{m}$ be an
isometric embedding. Then
\begin{equation}
\int_{\Sigma}H\leq\int_{\Sigma}\frac{\left\vert \overrightarrow{H}%
_{0}\right\vert ^{2}}{H},
\end{equation}
where $\overrightarrow{H}_{0}$ is the mean curvature vector of $\iota
:\Sigma\rightarrow\mathbb{R}^{m}$. Moreover, if equality holds, then
$\iota\left(  \Sigma\right)  $ is contained in an $n$-dimensional plane of
$\mathbb{R}^{m}$ and $M$ is isometric to the Euclidean domain enclosed by
$\iota\left(  \Sigma\right)  $ in that $n$-dimensional plane.
\end{theorem}

Notice that an isometric embedding $\iota:\Sigma\rightarrow\mathbb{R}^{m}$
always exists by Nash's famous theorem.

When the scalar curvature has a negative lower bound, results similar to
Theorem \ref{ST} were proved by \cite{WY1} and \cite{ST2}. The counterexample
to the Min-Oo conjecture by Brendle, Marques and Neves\cite{BMN} shows that no
such result holds when the scalar curvature has a positive lower bound.
Results similar to Theorem \ref{MW} are also established when the Ricci
curvature has a positive or negative lower bound. But in these two cases the
inequalities obtained are not sharp. Some rigidity results under stronger
assumptions on the boundary were proved in \cite{MW}.

In all of these studies the result is basically an estimate on an integral
involving the mean curvature. It is natural to ask if one can bound the area
of the boundary, the volume of the interior and other more direct geometric or
analytic quantities. It is easy to see that for such results to hold a lower
bound for the mean curvature is not enough. For example, for any closed
$\left(  \Sigma^{n-2},h\right)  $ with nonnegative Ricci curvature,
$M:=\overline{B^{2}}\times\Sigma$ with the product metric $dx^{2}+h$ has
nonnegative Ricci curvature and mean curvature $H\geq1$ while the area of
$\partial M$ can be arbitrarily large. Therefore we will in this paper mostly
consider compact Riemannian manifolds $\left(  M^{n},g\right)  $ with
nonnegative Ricci curvature and with a connected boundary $\Sigma$ whose 2nd
fundamental form has a positive lower bound. Motivated by a uniqueness theorem
in \cite{BVV}, we study positive harmonic functions on $M$ that satisfy a
semilinear Neumann condition on the boundary. We formulate a conjecture which
has important geometric implications. We will prove some partial results that
support this conjecture.\footnote{After this paper was posted to the arXiv,
new progress has been made in the following two papers: 1. Q. Guo and X. Wang,
Uniqueness results for positive harmonic functions on $\overline
{\mathbb{B}^{n}}$ satisfying a nonlinear boundary condition, arXiv:1912.05568.
2. Q. Guo, F. Hang and X. Wang, Liouville type theorems on manifolds with
nonnegative curvature and strictly convex boundary, arXiv:1912.05574} Another
case we consider is when $M$ has positive Ricci curvature, by scaling we can
always assume $Ric\geq n-1$ and the boundary $\Sigma$ is convex in the sense
that its 2nd fundamental form is nonnegative. There is similarly a natural
conjecture on the area of the boundary.

The paper is organized as follows. In section 2 we discuss some natural PDEs
on a compact manifold with boundary. We formulate a uniqueness conjecture on a
semilinear Neumann problem in the nonnegative Ricci case and discuss its
geometric implications. In Section 3 we prove some topological results. In
Section 4 we present some partial results and several other conjectures.

Acknowledgement.\textbf{ }The work of the author is partially supported by
Simons Foundation Collaboration Grant for Mathematicians \#312820.

\section{\bigskip from PDE to geometry: a conjecture}

We first recall a theorem proved by Bidaut-Veron and Veron \cite{BVV} \ .

\begin{theorem}
\label{bvv}(\cite{BVV} and \cite{I}) Let $\left(  M^{n},g\right)  $ be a
compact Riemannian manifold with a (possibly empty) convex boundary. Suppose
$u\in C^{\infty}\left(  M\right)  $ is a positive solution of the following
equation%
\[%
\begin{array}
[c]{ccc}%
-\Delta u+\lambda u=u^{q} & \text{on} & M,\\
\frac{\partial u}{\partial\nu}=0 & \text{on} & \partial M,
\end{array}
\]
where $\lambda>0$ is a constant and $1<q\leq\left(  n+2\right)  /\left(
n-2\right)  $. If $Ric\geq\frac{\left(  n-1\right)  \left(  q-1\right)
\lambda}{n}g$, then $u$ must be constant unless $q=\left(  n+2\right)
/\left(  n-2\right)  $ and $\left(  M^{n},g\right)  $ is isometric to $\left(
\mathbb{S}^{n},\frac{4\lambda}{n\left(  n-2\right)  }g_{0}\right)  $ or
$\left(  \mathbb{S}_{+}^{n},\frac{4\lambda}{n\left(  n-2\right)  }%
g_{0}\right)  $. In the latter case $u$ is given on $\mathbb{S}^{n}$ or
$\mathbb{S}_{+}^{n}$\ by the following formula%
\[
u=\frac{1}{\left(  a+x\cdot\xi\right)  ^{\left(  n-2\right)  /2}}.
\]
for some $\xi\in\mathbb{R}^{n+1}$ and some constant $a>\left\vert
\xi\right\vert $.
\end{theorem}

\ This theorem was proved by Bidaut-Veron and Veron \cite{BVV} when $\partial
M=\varnothing$ and by Ilias \cite{I} when $\partial M\neq\varnothing$ using
the same method. It has some important corollaries. We focus on the case
$\partial M\neq\varnothing$. We recall the Yamabe problem on a compact
Riemannian manifold $\left(  M^{n},g\right)  $ with boundary. The conformal
Laplacian is defined to be $L_{g}=-c_{n}\Delta_{g}+R_{g}$, with $c_{n}%
=\frac{4\left(  n-1\right)  }{n-2}$. If $\widetilde{g}=\phi^{4/\left(
n-2\right)  }g$, then
\[
L_{\widetilde{g}}u=\phi^{-\left(  n+2\right)  /\left(  n-2\right)  }%
L_{g}\left(  u\phi\right)  .
\]
Under the conformal deformation the mean curvature of the boundary transforms
according to the following formula%
\[
\frac{2\left(  n-1\right)  }{n-2}\frac{\partial\phi}{\partial\nu}%
+H\phi=\widetilde{H}\phi^{n/\left(  n-2\right)  }.
\]
We consider the following functional%
\begin{align*}
E_{g}\left(  u\right)   &  =\int_{M}c_{n}\left\vert \nabla u\right\vert
^{2}+Ru^{2}+2\int_{\partial M}Hu^{2}\\
&  =\int_{M}uL_{g}udv_{g}+\int_{\partial M}\left(  c_{n}\frac{\partial
u}{\partial\nu}+2Hu\right)  ud\sigma_{g}%
\end{align*}
This functional is conformally invariant: $E_{\widetilde{g}}\left(  u\right)
=E_{g}\left(  u\phi\right)  $. If $u$ is positive then%
\[
E_{g}\left(  u\right)  =\int_{M}R_{\widetilde{g}}dv_{\widetilde{g}}%
+\int_{\partial M}2H_{\widetilde{g}}d\sigma_{\widetilde{g}},
\]
where $\widetilde{g}=u^{4/\left(  n-2\right)  }g$.

\bigskip We define%
\[
\lambda\left(  M,g\right)  =\inf\frac{E_{g}\left(  u\right)  }{\int
_{M}\left\vert u\right\vert ^{2}}.
\]
The sign of $\lambda$ is conformally invariant. The Yamabe invariant is
defined to be%
\[
Y\left(  M,g\right)  =\inf\frac{E_{g}\left(  u\right)  }{\left(  \int
_{M}\left\vert u\right\vert ^{2n/\left(  n-2\right)  }\right)  ^{\left(
n-2\right)  /n}}.
\]
Aubin \cite{A} showed that $Y\left(  M,g\right)  \leq Y\left(  \mathbb{S}%
^{n}\right)  =n\left(  n-1\right)  \left(  \left\vert \mathbb{S}%
^{n}\right\vert \right)  ^{2/n}$ when $\partial M=\varnothing$ while Escobar
\cite{E1} and Cherrier \cite{C} proved that $Y\left(  M,g\right)  \leq
Y\left(  \mathbb{S}_{+}^{n}\right)  =n\left(  n-1\right)  \left(  \left\vert
\mathbb{S}^{n}\right\vert /2\right)  ^{2/n}$ when $\partial M\neq\varnothing$.

Let $\left(  M^{n},g\right)  $ be a compact Riemannian manifold with convex
boundary and $Ric\geq\left(  n-1\right)  $. From Theorem \ref{bvv} one can
derive the following

\begin{itemize}
\item (Sharp Sobolev inequalities) For $2<q\leq\left(  n+2\right)  /\left(
n-2\right)  $
\[
\left(  \frac{1}{V}\int_{M}\left\vert u\right\vert ^{q+1}\right)  ^{2/\left(
q+1\right)  }\leq\frac{q-1}{n}\frac{1}{V}\int_{M}\left\vert \nabla
u\right\vert ^{2}+\frac{1}{V}\int_{M}u^{2}.
\]

\item $Y\left(  M,g\right)  \geq n\left(  n-1\right)  V^{2/n}$. Moreover,
equality holds iff $g$ is Einstein with totally geodesic boundary.
\end{itemize}

This discussion also yields an analytic proof of the classic result that
$V\leq\left\vert \mathbb{S}^{n}\right\vert $ when $\partial M=\varnothing$ and
$V\leq\left\vert \mathbb{S}^{n}\right\vert /2$ when $\partial M\neq
\varnothing$.

Given a compact Riemannian problem $\left(  M^{n},g\right)  $ with nonempty
boundary, the type II Yamabe problem studied by Escobar \cite{E2} is whether
one can find a conformal metric $\widetilde{g}=\phi^{4/\left(  n-2\right)  }g$
with zero scalar curvature and constant mean curvature on the boundary. This
leads to the following equation%
\begin{align*}
L_{g}\phi &  =0\text{ on }M,\\
\frac{2\left(  n-1\right)  }{n-2}\frac{\partial\phi}{\partial\nu}+H\phi &
=c\phi^{n/\left(  n-2\right)  }\text{ on }\partial M.
\end{align*}
Assuming $\lambda\left(  M,g\right)  >0$ Escobar introduced the following
minimization%
\[
Q\left(  M,\partial M,g\right)  =\inf\frac{E_{g}\left(  u\right)  }{\left(
\int_{\partial M}\left\vert u\right\vert ^{2\left(  n-1\right)  /\left(
n-2\right)  }\right)  ^{\left(  n-2\right)  /\left(  n-1\right)  }}.
\]
Motivated by Theorem \ref{bvv} we propose to study \textbf{positive} solutions
of the following equation
\begin{equation}%
\begin{array}
[c]{ccc}%
\Delta u=0 & \text{on} & M,\\
\frac{\partial u}{\partial\nu}+\lambda u=u^{q} & \text{on} & \partial M,
\end{array}
\label{FE}%
\end{equation}
where $\lambda>0$ and $1<q\leq n/\left(  n-2\right)  $, and make the following conjecture.

\begin{conjecture}
\label{mainconj}Let $\left(  M^{n},g\right)  $ be a compact Riemannian
manifold with $Ric\geq0$ and $\Pi\geq1$ on $\partial M$. If $0<\lambda
\leq1/\left(  q-1\right)  $, then any positive solution $u$ of the above
equation must be constant unless $q=n/\left(  n-2\right)  $, $M$ is isometric
to $\overline{\mathbb{B}^{n}}\subset\mathbb{R}^{n}$ and $u$ corresponds to
\[
u_{a}\left(  x\right)  =\left[  \frac{2}{n-2}\frac{1-\left\vert a\right\vert
^{2}}{1+\left\vert a\right\vert ^{2}\left\vert x\right\vert ^{2}-2x\cdot
a}\right]  ^{\left(  n-2\right)  /2}%
\]
for some $a\in\mathbb{B}^{n}$.
\end{conjecture}

\bigskip At the moment this conjecture is completely open. But in dimension 2
an analogous problem was studied by the author \cite{W2} in which the
following result was proved.

\begin{theorem}
Let $(\Sigma,g)$ be a compact surface with Gaussian curvature $K\geq0$ and on
the boundary the geodesic curvature $\kappa\geq1$. Consider the following
equation%
\[%
\begin{array}
[c]{ccc}%
\Delta u=0 & \text{on} & \Sigma,\\
\frac{\partial u}{\partial\nu}+\lambda=e^{u} & \text{on} & \partial\Sigma,
\end{array}
\]
where $\lambda$ is a positive constant. If $\lambda<1$ then $u$ is constant;
if $\lambda=1$ and $u$ is not constant, then $\Sigma$ is isometric to the unit
disc $\overline{\mathbb{B}^{2}}$ and $u$ is given by
\[
u\left(  z\right)  =\log\frac{1-\left\vert a\right\vert ^{2}}{1+\left\vert
a\right\vert ^{2}\left\vert z\right\vert ^{2}-2\operatorname{Re}\left(
z\overline{a}\right)  },
\]
for some $a\in\mathbb{B}^{2}$.
\end{theorem}

Next we discuss a geometric implication of Conjecture \ref{mainconj}. For
$1<q<n/\left(  n-2\right)  $ the following minimization problem%
\[
\inf\frac{\left(  q-1\right)  \int_{M}\left\vert \nabla u\right\vert ^{2}%
+\int_{\partial M}u^{2}}{\left(  \int_{\partial M}\left\vert u\right\vert
^{q+1}\right)  ^{2/\left(  q+1\right)  }}%
\]
is achieved by smooth positive function satisfying (\ref{FE}) with
$\lambda=1/(q-1)$. If the conjecture is true then the minimizer is constant
and therefore the following inequality holds%

\begin{equation}
\left\vert \partial M\right\vert ^{\left(  q-1\right)  /\left(  q+1\right)
}\left(  \int_{\partial M}\left\vert u\right\vert ^{q+1}\right)  ^{2/\left(
q+1\right)  }\leq\left(  q-1\right)  \int_{M}\left\vert \nabla u\right\vert
^{2}+\int_{\partial M}u^{2}. \label{ineq_t}%
\end{equation}
Letting $q\nearrow n/\left(  n-2\right)  $ yields%
\[
\left\vert \partial M\right\vert ^{1/\left(  n-1\right)  }\left(
\int_{\partial M}\left\vert u\right\vert ^{2\left(  n-1\right)  /\left(
n-2\right)  }\right)  ^{\left(  n-2\right)  /\left(  n-1\right)  }\leq\int
_{M}\frac{2}{n-2}\left\vert \nabla u\right\vert ^{2}+\int_{\partial M}u^{2}.
\]
Then
\begin{align*}
E_{g}\left(  u\right)   &  =\int_{M}\frac{4\left(  n-1\right)  }%
{n-2}\left\vert \nabla u\right\vert ^{2}+Ru^{2}+2\int_{\partial M}Hu^{2}\\
&  \geq\int_{M}\frac{4\left(  n-1\right)  }{n-2}\left\vert \nabla u\right\vert
^{2}+2\left(  n-1\right)  \int_{\partial M}u^{2}\\
&  \geq2\left(  n-1\right)  \left\vert \partial M\right\vert ^{1/\left(
n-1\right)  }\left(  \int_{\partial M}\left\vert u\right\vert ^{2\left(
n-1\right)  /\left(  n-2\right)  }\right)  ^{\left(  n-2\right)  /\left(
n-1\right)  }.
\end{align*}
Therefore
\[
Q\left(  M,\partial M,g\right)  \geq2\left(  n-1\right)  \left\vert \partial
M\right\vert ^{1/\left(  n-1\right)  }.
\]
As $Q\left(  M,\partial M,g\right)  \leq Q\left(  \mathbb{B}^{n}%
,\partial\mathbb{B}^{n}\right)  =2\left(  n-1\right)  \left\vert
\mathbb{S}^{n-1}\right\vert ^{1/\left(  n-1\right)  }$ we obtain%
\[
\left\vert \partial M\right\vert \leq\left\vert \mathbb{S}^{n-1}\right\vert .
\]
In summary Conjecture \ref{mainconj} implies the following conjecture.

\begin{conjecture}
\label{bsnn}Let $\left(  M^{n},g\right)  $ be a compact Riemannian manifold
with $Ric\geq0$ and $\Pi\geq1$ on $\partial M$. Then
\[
\left\vert \partial M\right\vert \leq\left\vert \mathbb{S}^{n-1}\right\vert .
\]

\end{conjecture}

We remark that the inequality (\ref{ineq_t}) for $1\leq q\leq n/\left(
n-2\right)  $ on a compact Riemannian manifold $\left(  M^{n},g\right)  $ with
$Ric\geq0$ and $\Pi\geq1$ on $\partial M$ that would follow from Conjecture
\ref{mainconj} is known to be true on $\overline{\mathbb{B}^{n}}$. This was
proved by Beckner \cite{B} as a corollary of the Hardy-Littlewood-Sobolev
inequality with sharp constant on the sphere. Here is the precise statement

\begin{theorem}
(Beckner \cite{B}) For $1\leq q\leq n/\left(  n-2\right)  $
\[
c_{n}{}^{\left(  q-1\right)  /\left(  q+1\right)  }\left(  \int_{\mathbb{S}%
^{n-1}}\left\vert F(\xi)\right\vert ^{q+1}d\xi\right)  ^{2/\left(  q+1\right)
}\leq\left(  q-1\right)  \int_{\mathbb{B}^{n}}\left\vert \nabla
u(x)\right\vert ^{2}dx+\int_{\mathbb{S}^{n}}\left\vert F(\xi)\right\vert
^{2}d\xi,
\]
where $u$ is the harmonic extension of $F$ and $c_{n}=2\pi^{n/2}/\Gamma\left(
n/2\right)  =\left\vert \mathbb{S}^{n-1}\right\vert $.
\end{theorem}

\section{Boundary effect on topology}

In this section we prove some topological results on compact Riemannian
manifolds with a lower bound for Ricci curvature and a corresponding lower
bound for the 2nd fundamental form on the boundary. Besides their independent
interest, these topological results will be used to prove some geometric
results in the next section.

\begin{proposition}
Let $\left(  M^{n},g\right)  $ be a compact Riemannian manifold with boundary
$\Sigma$. Suppose $Ric\geq0$.

\begin{itemize}
\item If the boundary has positive mean curvature, then $H^{1}\left(
M,\Sigma\right)  =0$.

\item If the boundary is strictly convex, then $H^{1}\left(  M\right)  =0$.
\end{itemize}
\end{proposition}

This result should be well known. A proof using minimal surfaces for the 2nd
part was given by Fraser and Li \cite{FL}. We explain the standard argument
with harmonic forms. By the Hodge theory for compact Riemannian manifolds with
boundary
\[
H^{1}\left(  M,\Sigma\right)  \cong\mathcal{H}_{R}^{1}\left(  M\right)  ,
\]
where $\mathcal{H}_{R}^{1}\left(  M\right)  $ is the space of harmonic
$1$-forms satisfying the relative boundary condition, i.e. $\alpha
\in\mathcal{H}_{R}^{1}\left(  M\right)  $ iff $d\alpha=0,d^{\ast}\alpha=0$ and
$\alpha\wedge\nu^{\ast}=0$ on the boundary. Note that the boundary condition
simply means $\alpha\left(  e_{i}\right)  =0,i=1,\cdots,n-1$. Thus
$\left\langle \nabla_{\nu}\alpha,\alpha\right\rangle =\alpha\left(
\nu\right)  \nabla_{\nu}\alpha\left(  \nu\right)  $ on the boundary. We
compute
\begin{align*}
\nabla_{\nu}\alpha\left(  \nu\right)   &  =-\sum_{i=1}^{n-1}\nabla_{e_{i}%
}\alpha\left(  e_{i}\right) \\
&  =\sum_{i=1}^{n-1}\left(  -e_{i}\left(  \alpha\left(  e_{i}\right)  \right)
+\alpha\left(  \nabla_{e_{i}}e_{i}\right)  \right) \\
&  =-H\alpha\left(  \nu\right)  .
\end{align*}
By the Bochner formula we have%
\begin{align*}
\int_{M}\left\vert \nabla\alpha\right\vert ^{2}+Ric\left(  \alpha
,\alpha\right)   &  =\int_{\Sigma}\left\langle \nabla_{\nu}\alpha
,\alpha\right\rangle \\
&  =-\int_{\Sigma}H\left[  \alpha\left(  \nu\right)  \right]  ^{2}.
\end{align*}
Clearly $\alpha=0$ if $Ric\geq0$ and $H>0$. Therefore $H^{1}\left(
M,\Sigma\right)  =0$.

For the second part we recall
\[
H^{1}\left(  M\right)  \cong\mathcal{H}_{A}^{1}\left(  M\right)  ,
\]
where $\mathcal{H}_{A}^{1}\left(  M\right)  $ is the space of harmonic
$1$-forms satisfying the absolute boundary condition, i.e. $\alpha
\in\mathcal{H}_{A}^{1}\left(  M\right)  $ iff $d\alpha=0,d^{\ast}\alpha=0$ and
$\alpha\left(  \nu\right)  =0$ on the boundary. Working with a local
orthonormal frame $\left\{  e_{0}=\nu,e_{1},\cdots,e_{n-1}\right\}  $ on
$\Sigma$ we have%
\begin{align*}
\left\langle \nabla_{\nu}\alpha,\alpha\right\rangle  &  =\sum_{i=1}%
^{n-1}\alpha\left(  e_{i}\right)  \nabla_{\nu}\alpha\left(  e_{i}\right) \\
&  =\sum_{i=1}^{n-1}\alpha\left(  e_{i}\right)  \nabla_{e_{i}}\alpha\left(
\nu\right) \\
&  =\sum_{i=1}^{n-1}\alpha\left(  e_{i}\right)  \left[  e_{i}\left(
\alpha\left(  \nu\right)  \right)  -\alpha\left(  \nabla_{e_{i}}\nu\right)
\right] \\
&  =-\sum_{i,j=1}^{n-1}\Pi_{ij}\alpha\left(  e_{i}\right)  \alpha\left(
e_{j}\right)  .
\end{align*}
Therefore%
\[
\int_{M}\left\vert \nabla\alpha\right\vert ^{2}+Ric\left(  \alpha
,\alpha\right)  =-\int_{\Sigma}\sum_{i,j=1}^{n-1}\Pi_{ij}\alpha\left(
e_{i}\right)  \alpha\left(  e_{j}\right)  .
\]
Since $Ric\geq0$ and $\Pi>0$, we must have $\alpha=0$. Therefore $H^{1}\left(
M\right)  =0$.

\begin{remark}
\label{pge0}In the second part if we only assume $\Pi\geq0$, then the same
argument proves that a harmonic form $\alpha\in\mathcal{H}_{A}^{1}\left(
M\right)  $ must be parallel. As $\alpha\left(  \nu\right)  =0$ on the
boundary $\Sigma$ we can write $\alpha=\left\langle X,\cdot\right\rangle $ on
$\Sigma$, where $X$ is a vector field on $\Sigma$. As $\alpha$ is parallel it
is easy to see that $X$ is a parallel vector field on $\Sigma$. Therefore we
conclude that either $H^{1}\left(  M\right)  =0$ or there exists a nonzero
parallel vector field on $\Sigma$.
\end{remark}

In dimension 3 we have the following consequence.

\begin{corollary}
\label{top3}Let $\left(  M^{3},g\right)  $ be a compact Riemannian
$3$-manifold with boundary $\Sigma$. Suppose $Ric\geq0$ and the boundary is
strictly convex. Then the boundary $\Sigma$ is topologically a sphere.
\end{corollary}

\begin{proof}
We have the long exact sequence%
\[
\cdots\rightarrow H^{1}\left(  M,\Sigma\right)  \rightarrow H^{1}\left(
M\right)  \rightarrow H^{1}\left(  \Sigma\right)  \rightarrow H^{2}\left(
M,\Sigma\right)  \rightarrow\cdots
\]
By Poincare duality $H^{2}\left(  M,\Sigma\right)  \approx H^{1}\left(
M\right)  $. Since $H^{1}\left(  M\right)  =0$ we must have $H^{1}\left(
\Sigma\right)  =0$, i.e. $\Sigma$ is topologically a sphere.
\end{proof}

In fact the same argument combined with Remark \ref{pge0} yields

\begin{proposition}
Let $\left(  M^{3},g\right)  $ be a compact Riemannian $3$-manifold with
boundary $\Sigma$. Suppose $Ric\geq0$ and the boundary is convex. Then
$\Sigma$ is either a topological sphere or a flat torus.
\end{proposition}

Therefore the boundary cannot be a Riemann surface of higher genus.

The above two results in dimension 3 may be deduced from the work of
Meeks-Simon-Yau \cite{MSY}, but the argument here is much more elementary.

\bigskip The same argument works for the following situation.

\begin{proposition}
\label{vanishing_nr}Let $\left(  M^{n},g\right)  $ be a compact Riemannian
manifold with positive Ricci curvature and convex boundary. If the boundary is
convex, then both $H^{1}\left(  M,\Sigma\right)  $ and $H^{1}\left(  M\right)
$ vanish.
\end{proposition}

When the Ricci curvature has a negative lower bound, we can also prove the
vanishing of $H^{1}\left(  M,\Sigma\right)  $ if the boundary has sufficiently
large mean curvature.

\begin{proposition}
Let $\left(  M^{n},g\right)  $ be a compact Riemannian manifold with boundary
$\Sigma$. Suppose $Ric\geq-\left(  n-1\right)  $. If $H\geq\left(  n-1\right)
$, then $H^{1}\left(  M,\Sigma\right)  =0$.
\end{proposition}

The proof is more complicated. \ We first recall the following result which
can be proved by classic methods.

\begin{proposition}
Let $\left(  M^{n},g\right)  $ be a compact Riemannian manifold with
$Ric\geq-\left(  n-1\right)  $. Let $\rho$ be the distance function to the
boundary. Suppose the mean curvature of the boundary satisfies $H\geq\left(
n-1\right)  $. Then in the support sense%
\[
\Delta\rho\leq-\left(  n-1\right)  .
\]

\end{proposition}

We compute%
\begin{align*}
\Delta e^{c\rho}  &  =e^{c\rho}\left[  c\Delta\rho+c^{2}\left\vert \nabla
\rho\right\vert ^{2}\right] \\
&  \leq e^{c\rho}\left[  -\left(  n-1\right)  c+c^{2}\right] \\
&  =e^{c\rho}\left[  \left(  c-\frac{n-1}{2}\right)  ^{2}-\frac{\left(
n-1\right)  ^{2}}{4}\right]  .
\end{align*}
Let $\phi=e^{\left(  n-1\right)  \rho/2}$. We have%
\begin{equation}
\Delta\phi\leq-\frac{\left(  n-1\right)  ^{2}}{4}\phi. \label{lapf}%
\end{equation}
It is well know that this implies that the first Dirichlet eigenvalue
$\lambda_{1}\geq\frac{\left(  n-1\right)  ^{2}}{4}$.

We now prove the first part of Proposition \ref{vanishing_nr}. Let $\alpha
\in\mathcal{H}_{R}^{1}\left(  M\right)  $. By a computation due to Yau we
have
\[
\left\vert \nabla\alpha\right\vert ^{2}\geq\frac{n}{n-1}\left\vert
\nabla\left\vert \alpha\right\vert \right\vert ^{2}.
\]
By the Bochner formula we have%
\begin{align*}
\frac{1}{2}\Delta\left\vert \alpha\right\vert ^{2}  &  =\left\vert
\nabla\alpha\right\vert ^{2}+Ric\left(  \alpha,\alpha\right) \\
&  \geq\frac{n}{n-1}\left\vert \nabla\left\vert \alpha\right\vert \right\vert
^{2}-\left(  n-1\right)  \left\vert \alpha\right\vert ^{2}.
\end{align*}
Therefore
\[
\left\vert \alpha\right\vert \Delta\left\vert \alpha\right\vert \geq\frac
{1}{n-1}\left\vert \nabla\left\vert \alpha\right\vert \right\vert ^{2}-\left(
n-1\right)  \left\vert \alpha\right\vert ^{2}.
\]
Let $f=\left\vert \alpha\right\vert ^{\left(  n-2\right)  /\left(  n-1\right)
} $. Direct calculation yields%
\[
\Delta f\geq-\left(  n-2\right)  f.
\]
Let $u=f/\phi$. Direct calculation yields%
\begin{align*}
\Delta u  &  \geq\left[  \frac{\left(  n-1\right)  ^{2}}{4}-\left(
n-2\right)  \right]  u-2\phi^{-1}\left\langle \nabla u,\nabla\phi\right\rangle
\\
&  =\frac{\left(  n-3\right)  ^{2}}{4}u-2\phi^{-1}\left\langle \nabla
u,\nabla\phi\right\rangle
\end{align*}
Suppose that $f$ is not identically zero. By the maximum principle $u$ must
achieve its positive maximum somewhere on the boundary and furthermore at this
point we must have
\[
\frac{\partial u}{\partial\nu}\geq0.
\]
On the other hand on the boundary, as $H\geq n-1$%
\begin{align*}
\frac{\partial u}{\partial\nu}  &  =u\left[  \frac{n-2}{n-1}\frac{\left\langle
\nabla_{\nu}\alpha,\alpha\right\rangle }{\left\vert \alpha\right\vert ^{2}%
}+\frac{n-1}{2}\right] \\
&  =u\left[  -\frac{n-2}{n-1}H+\frac{n-1}{2}\right] \\
&  \leq u\left[  -\left(  n-2\right)  +\frac{n-1}{2}\right] \\
&  =-u\left(  n-3\right)  /2.
\end{align*}
This is strictly negative when $n\geq4$ and hence a contradiction. Therefore
$\alpha$ is identically zero. When $n=3$ and if $\alpha$ is not identically
zero, then by the Hopf lemma $u$ must be a positive constant. By scaling
$\alpha$ we can assume that $u\equiv1$ or $\phi=f$. Therefore $-\Delta
\phi=\phi$. By elliptic regularity $\phi$ is smooth. From the proof of
\ (\ref{lapf}) it follows that $\rho$ is smooth everywhere and $\left\vert
\nabla\rho\right\vert \equiv1$. But this is impossible as $\rho$ is not smooth
at a cut point, e.g. at a point where it achieves its maximum. Therefore we
must have $\alpha=0$ too when $n=3$.

\bigskip The same argument can be used to prove the following: Let $\left(
M^{n},g\right)  $ be a compact Riemannian manifold with boundary $\Sigma$.
Suppose $Ric\geq-\left(  n-1\right)  $. If the second fundamental form of
$\Sigma$ satisfies $\Pi>\left(  n-1\right)  /\sqrt{2\left(  n-2\right)  }$,
then $H^{1}\left(  M\right)  =0$. The only difference is that at the end we
have for $\alpha\in\mathcal{H}_{A}^{1}\left(  M\right)  $%
\begin{align*}
\frac{\partial u}{\partial\nu}  &  =u\left[  \frac{n-2}{n-1}\frac{\left\langle
\nabla_{\nu}\alpha,\alpha\right\rangle }{\left\vert \alpha\right\vert ^{2}%
}+\frac{n-1}{2}\frac{\sinh R}{\cosh R}\right] \\
&  =u\left[  -\frac{n-2}{n-1}\frac{\Pi_{ij}\alpha\left(  e_{i}\right)
\alpha\left(  e_{i}\right)  }{\left\vert \alpha\right\vert ^{2}}+\frac{n-1}%
{2}\frac{\sinh R}{\cosh R}\right] \\
&  \leq u\left[  -\frac{n-2}{n-1}\frac{\cosh R}{\sinh R}+\frac{n-1}{2}%
\frac{\sinh R}{\cosh R}\right]  .
\end{align*}
Is the constant sharp? It seems reasonable to expect $H^{1}\left(  M\right)
=0 $ if $\Pi>1$.

\bigskip

\bigskip

\section{\bigskip On the size of the boundary}

In this section we prove some estimates on the size of the boundary, in
particular we show that Conjecture \ref{bsnn} is true in dimension 3. First we
recall a result in Xia \cite{X}.

\begin{proposition}
\label{Reilly}Let $(M,g)$ be a compact Riemannian manifold with boundary
$\Sigma$. Suppose $Ric(g)\geq0$ and $\Pi\geq1$. Then $\lambda_{1}(\Sigma)\geq
n-1$ and the equality holds iff $(M,g)$ is isometric to the unit ball in
Euclidean space $%
%TCIMACRO{\U{211d} }%
%BeginExpansion
\mathbb{R}
%EndExpansion
^{n}$.
\end{proposition}

The proof is based on Reilly's formula \cite{Re}. For completeness and
comparison later, we present the proof. Let $u$ be the solution of the
following equation%
\[
\left\{
\begin{array}
[c]{c}%
\Delta u=0\text{ \ \ on }M,\\
u|_{\Sigma}=f
\end{array}
\right.
\]
where $f$ is a first eigenfunction on $\Sigma$, i.e. $-\bigtriangleup_{\Sigma
}f=\lambda_{1}f$. Let $\chi=\frac{\partial u}{\partial\nu}$ with $\nu$ being
the outer unit normal. By Reilly's formula%
\begin{align*}
&  \int_{M}\left(  \Delta u\right)  ^{2}-\left\vert D^{2}u\right\vert
^{2}-Ric\left(  \nabla u,\nabla u\right) \\
&  =\int_{\Sigma}\left[  2\chi\Delta_{\Sigma}f+H\chi^{2}+\Pi\left(  \nabla
f,\nabla f\right)  \right]  d\sigma\\
&  \geq\int_{\Sigma}-2\lambda_{1}\chi f+\left(  n-1\right)  \chi
^{2}+\left\vert \nabla f\right\vert ^{2}%
\end{align*}
As $\Delta u=0$ and $\int_{\Sigma}\left\vert \nabla f\right\vert ^{2}%
=\lambda_{1}\int_{\Sigma}f^{2}$%
\[
\int_{\Sigma}-2\lambda_{1}\chi f+\left(  n-1\right)  \chi^{2}+\lambda_{1}%
f^{2}\leq0,
\]
whence%
\[
\frac{\lambda_{1}(\lambda_{1}-n+1)}{n-1}\int_{\partial M}f^{2}\geq
\int_{\partial M}(\chi-\frac{\lambda_{1}}{n-1}f)^{2}\geq0.
\]
Therefore $\lambda_{1}\geq n-1$.

If $\lambda_{1}=n-1$, then we must have $D^{2}u=0,\chi=f$ and $\Pi=g|_{\Sigma
}$. As a consequence we have $D^{2}f=-fg$ on $\partial\Sigma$. By the well
known Obata theorem $\partial M$ is isometric to the standard sphere $S^{n-1}%
$. Let $f_{1},\cdots,f_{n}$ be a standard basis of the first eigenspace on
$\partial M$ $\cong S^{n-1}$ and $u_{1},\cdots,u_{n}$ the corresponding
harmonic extensions on $M$. We know that $\nabla u_{1,}\cdots,\nabla u_{n}$
are parallel vector fields on $M$. It is then easy to see that $U=(u_{1}%
,\cdots,u_{n})$ isometrically embeds $M$ into $%
%TCIMACRO{\U{211d} }%
%BeginExpansion
\mathbb{R}
%EndExpansion
^{n}$ with the image the unit ball.

This is basically the same argument used by Choi and Wang \cite{CW} to prove
that the 1st eigenvalue of \ an embedded minimal hypersurface $\Sigma
^{n-1}\subset$ $\mathbb{S}^{n}$ is at least $\left(  n-1\right)  /2$. But a
conjecture of Yau the 1st eigenvalue should equal to $n-1$. By the same
argument we have the following estimate for a general compact Riemannian
manifold with $Ric\geq n$ and with a convex boundary.

\begin{proposition}
\label{lam+}Let $(M^{n},g)$ be a compact Riemannian manifold with boundary
$\Sigma$. Suppose $Ric(g)\geq n-1$ and $\Pi\geq0$. Then $\lambda_{1}%
(\Sigma)\geq\left(  n-1\right)  /2$.
\end{proposition}

\begin{remark}
In view of Yau's conjecture, we also conjecture that in this case the best
lower bound is $n-1$.
\end{remark}

\begin{proof}
Let $u$ be the solution of the following equation%
\begin{equation}
\left\{
\begin{array}
[c]{c}%
\Delta u=0\text{ \ \ on }M,\\
u|_{\Sigma}=f
\end{array}
\right.
\end{equation}
where $f$ is a first eigenfunction on $\Sigma$, i.e. $-\bigtriangleup_{\Sigma
}f=\lambda_{1}f$. Let $\chi=\frac{\partial u}{\partial\nu}$ with $\nu$ being
the outer unit normal. By Reilly's formula%
\begin{align*}
&  \int_{M}\left(  \Delta u\right)  ^{2}-\left\vert D^{2}u\right\vert
^{2}-Ric\left(  \nabla u,\nabla u\right) \\
&  =\int_{\Sigma}\left[  2\chi\Delta_{\Sigma}f+H\chi^{2}+\Pi\left(  \nabla
f,\nabla f\right)  \right]  d\sigma\\
&  \geq\int_{\Sigma}-2\lambda_{1}\chi f,
\end{align*}
as $\Pi\geq0$. Thus we get%
\[
2\lambda_{1}\int_{\Sigma}\chi f\geq\left(  n-1\right)  \int_{M}\left\vert
\nabla u\right\vert ^{2}%
\]
From the equation of $u$ we have $\int_{\Sigma}\left\vert \nabla u\right\vert
^{2}=\int_{\Sigma}f\chi$. Thus%
\[
\left[  2\lambda_{1}-\left(  n-1\right)  \right]  \int_{M}\left\vert \nabla
u\right\vert ^{2}\geq0,
\]
Therefore $\lambda_{1}\geq\left(  n-1\right)  /2$.
\end{proof}

\bigskip We will also need the following result due to Ros \cite{Ros}, which
was also proved by Reilly's formula.

\begin{theorem}
\label{Ros}(Ros) Let $(M,g)$ be a compact Riemannian manifold with boundary.
If $Ric\geq0$ and the mean curvature $H$ of $\partial M$ is positive, then%
\[
\int_{\partial M}\frac{1}{H}d\sigma\geq\frac{n}{n-1}V.
\]
The equality holds iff $M$ is isometric to an Euclidean ball.
\end{theorem}

We can now prove the following result in dimension 3.

\begin{theorem}
\label{nn3}Let $(M^{3},g)$ be a compact Riemannian manifold with boundary
$\Sigma$. Suppose $Ric(g)\geq0$ and $\Pi\geq g|_{\Sigma}$. Then

\begin{itemize}
\item $A\left(  \Sigma\right)  \leq4\pi$;

\item $V\left(  M\right)  \leq4\pi/3$.
\end{itemize}

Moreover if equality holds in either case, $M$ is isometric to the unit ball
$\overline{\mathbb{B}^{3}}\subset\mathbb{R}^{3}$.
\end{theorem}

\begin{proof}
By Proposition \ref{Reilly} we have $\lambda_{1}\left(  \Sigma\right)  \geq2
$. By Corollary \ref{top3} $\Sigma$ is topologically $\mathbb{S}^{2}$. Then by
a theorem of Hersch \cite{H} (see also \cite[page 135]{SY}) $A\left(
\Sigma\right)  \leq8\pi/\lambda_{1}\left(  \Sigma\right)  $ and moreover
equality holds iff $\Sigma$ is a round sphere. By Proposition \ref{Reilly} we
have $\lambda_{1}\left(  \Sigma\right)  \geq2$. Therefore $A\left(
\Sigma\right)  \leq4\pi$. If equality holds, then $\lambda_{1}\left(
\Sigma\right)  =2$ and hence $M$ is isometric to $\overline{\mathbb{B}^{3}}$
by the rigidity part of Proposition \ref{Reilly}.

The 2nd part easily follows from combining the first part and Theorem
\ref{Ros}.
\end{proof}

\begin{example}
Let $\left(  S^{n-2},h\right)  $ be compact Riemannian manifold with
nonnegative Ricci curvature. Then $\overline{B^{2}}\times S$ has nonnegative
Ricci curvature and the boundary has mean curvature $H=1$. This show that the
conjecture is not true if the condition on 2nd fundamental form is weakened to
a condition on the mean curvature.
\end{example}

\bigskip In the case of positive Ricci curvature we make the following

\begin{conjecture}
\label{conj+}Let $\left(  M^{n},g\right)  $ be a compact Riemannian manifold
with $Ric\geq n-1$ and $\Pi\geq0$ on $\Sigma=\partial M$. Then
\[
\left\vert \Sigma\right\vert \leq\left\vert \mathbb{S}^{n-1}\right\vert .
\]
Moreover if equality holds then $\left(  M^{n},g\right)  $ is isometric to the
hemisphere $\mathbb{S}_{+}^{n}=\{x\in\mathbb{R}^{n+1}:|x|=1,x_{n+1}%
\geq0\}\subset\mathbb{R}^{n+1}$.
\end{conjecture}

\bigskip In \cite{HW} the following rigidity result was established.

\begin{theorem}
\label{HWr}Let $\left(  M^{n},g\right)  $ ($n\geq2$) be a compact Riemannian
manifold with nonempty boundary $\Sigma=\partial M$. Suppose

\begin{itemize}
\item \textrm{$Ric$}$\geq\left(  n-1\right)  g,$

\item $\left(  \Sigma,g|_{\Sigma}\right)  $ is isometric to the standard
sphere $\mathbb{S}^{n-1}\subset\mathbb{R}^{n}$,

\item $\Sigma$ is convex in $M$ in the sense that its second fundamental form
is nonnegative.
\end{itemize}

Then $\left(  M,g\right)  $ is isometric to the hemisphere $\mathbb{S}_{+}%
^{n}$.
\end{theorem}

Therefore the conjecture, if true, is a far-reaching generalization of the
above rigidity result. When $n=2$ the above theorem can be reformulated as follows.

\begin{theorem}
Let $(M^{2},g)$ be compact surface with boundary and the Gaussian curvature
$K\geq1.$ Suppose the geodesic curvature $k$ of the boundary $\gamma$
satisfies $k\geq0$. Then $L(\gamma)\leq2\pi$. Moreover equality holds iff
$(M,g)$ is isometric to $\mathbb{S}_{+}^{2}$.
\end{theorem}

It implies a classic result of Toponogov \cite{T}:

\begin{quotation}
Let $(M^{2},g)$ be a closed surface with Gaussian curvature $K\geq1$. Then any
simple closed geodesic in $M$ has length at most $2\pi$. Moreover if there is
one with length $2\pi$, then $M$ is isometric to the standard sphere
$\mathbb{S}^{2}$.
\end{quotation}

We refer to \cite{HW} for more details. In view of this connection, Conjecture
\ref{conj+} can be viewed as a generalization of Toponogov's theorem in higher
dimensions. We note that Marques and Neves \cite{MN} offered a generalization
of Toponogov's theorem in dimension 3 in terms of the scalar curvature.

As an evidence for Conjecture \ref{conj+} we show that it is true under the
stronger condition that sectional curvatures are at least one.

\begin{proposition}
Let $\left(  M^{n},g\right)  $ be a compact Riemannian manifold with $\sec
\geq1$ and $\Pi\geq0$ on $\Sigma=\partial M$. Then
\[
\left\vert \Sigma\right\vert \leq\left\vert \mathbb{S}^{n-1}\right\vert .
\]
Moreover if equality holds then $\left(  M^{n},g\right)  $ is isometric to the
hemisphere $\mathbb{S}_{+}^{n}=\{x\in\mathbb{R}^{n+1}:|x|=1,x_{n+1}%
\geq0\}\subset\mathbb{R}^{n+1}$.
\end{proposition}

\begin{proof}
The proof of the inequality is elementary. By the Gauss equation for any
orthonormal pair $X,Y\in T_{p}\Sigma$%
\begin{align*}
R^{\Sigma}\left(  X,Y,X,Y\right)   &  =R^{\Sigma}\left(  X,Y,X,Y\right)
+\Pi\left(  X,X\right)  \Pi\left(  Y,Y\right)  -\Pi\left(  X,Y\right)  ^{2}\\
&  \geq1+\Pi\left(  X,X\right)  \Pi\left(  Y,Y\right)  -\Pi\left(  X,Y\right)
^{2}.
\end{align*}
Since $\Pi\geq0$ it is a simple algebraic fact that $\Pi\left(  X,X\right)
\Pi\left(  Y,Y\right)  -\Pi\left(  X,Y\right)  ^{2}\geq0$. Therefore
$R^{\Sigma}\left(  X,Y,X,Y\right)  \geq1$, i.e. $\sec_{\Sigma}\geq1$. By the
Bishop-Gromov volume comparison we have $\left\vert \Sigma\right\vert
\leq\left\vert \mathbb{S}^{n-1}\right\vert $.

Moreover if $\left\vert \Sigma\right\vert =\left\vert \mathbb{S}%
^{n-1}\right\vert $, then $\Sigma$ is isometric to $\mathbb{S}^{n-1}$. By
Theorem \ref{HWr} $M$ is isometric to the hemisphere $\mathbb{S}_{+}^{n}$.
\end{proof}

Similarly we have the following parallel result when sectional curvature is nonnegative.

\begin{proposition}
Let $\left(  M^{n},g\right)  $ be a compact Riemannian manifold with $\sec
\geq0$ and $\Pi\geq1$ on $\Sigma=\partial M$. Then
\[
\left\vert \Sigma\right\vert \leq\left\vert \mathbb{S}^{n-1}\right\vert .
\]
Moreover if equality holds then $\left(  M^{n},g\right)  $ is isometric to the
hemisphere $\overline{\mathbb{B}^{n}}=\{x\in\mathbb{R}^{n}:|x|\leq1\}$.
\end{proposition}

\bigskip Using Proposition \ref{lam+} one can easily prove the following by
the same method used to prove Theorem \ref{nn3}.

\begin{proposition}
Let $(M^{3},g)$ be a compact Riemannian manifold with boundary $\Sigma$.
Suppose $Ric(g)\geq2$ and $\Pi\geq0$. Then $A\left(  \Sigma\right)  \leq8\pi$
\end{proposition}

As stated in Conjecture \ref{conj+} the optimal upper bound should be $4\pi$.

\end{document}